\documentclass[a4paper,10pt]{article}
\usepackage{amsmath,amsthm,amssymb,amsfonts}
\usepackage{enumerate,amscd,amsxtra,MnSymbol}
\usepackage{mathrsfs}
\usepackage{color}
\usepackage[cmtip,all]{xy}

\newtheorem{theorem}{Theorem}[section]
\newtheorem*{theorem*}{Theorem}
\newtheorem{lemma}[theorem]{Lemma}
\newtheorem{proposition}[theorem]{Proposition}
\newtheorem*{proposition*}{Proposition}
\newtheorem{corollary}[theorem]{Corollary}
\newtheorem*{corollary*}{Corollary}

\newtheorem{definition*}{Definition}
\newtheorem{remark}[theorem]{Remark}
\newtheorem{remark*}{Remark}

\newtheorem{conjecture*}{Conjecture}

\newcommand{\bA}{{\mathbb A}}

\newcommand{\caM}{{\mathcal M}}

\newcommand{\caO}{{\mathcal O}}
\newcommand{\caH}{{\mathcal H}}

\renewcommand{\P}{{\mathbb P}}

\newcommand{\D}{{\mathsf D}}

\newcommand{\naturals}{{\mathbb N}}
\newcommand{\integers}{{\mathbb Z}}

\newcommand{\Hom}{\mathrm{Hom}}

\newcommand{\MGL}{\mathsf{MGL}}

\newcommand{\SH}{\mathsf{SH}}

\newcommand{\KGL}{\mathsf{KGL}}

\newcommand{\MU}{\mathsf{MU}}
\newcommand{\MZ}{\mathsf{M}\mathbb{Z}}
\newcommand{\MA}{\mathsf{M}A}
\newcommand{\MR}{\mathsf{M}R}

\newcommand{\Sm}{\mathrm{Sm}}

\newcommand{\Spec}{\mathrm{Spec}}

\newcommand{\f}{\mathsf{f}}

\newcommand{\Sh}{\mathrm{Sh}}

\newcommand{\hocolim}{\mathrm{hocolim}}
\newcommand{\holim}{\mathrm{holim}}

\newcommand{\h}{\mathsf{h}}

\newcommand{\Zar}{\mathit{Zar}}
\newcommand{\Nis}{\mathit{Nis}}

\newcommand{\GmS}{\mathbb{G}_{m,S}}

\newcommand{\Pic}{\mathrm{Pic}}

\newcommand{\upi}{\underline{\pi}}
\newcommand{\pre}{\mathrm{pre}}
\newcommand{\se}{{S^1}}

\title{Algebraic Cobordism in mixed characteristic}
\author{Markus Spitzweck}

\begin{document}

\maketitle

\begin{abstract}
We compute the geometric part of algebraic cobordism over Dedekind domains
of mixed characteristic after inverting the positive residue characteristics
and prove cases of a Conjecture of Voevodsky relating this geometric
part to the Lazard ring for regular local bases.
The method is by analyzing the slice tower of algebraic cobordism,
relying on the Hopkins-Morel isomorphism from the quotient
of the algebraic cobordism spectrum by the generators of the Lazard ring
to the motivic Eilenberg-MacLane spectrum, again after inverting the positive
residue characteristics.
\end{abstract}

\section{Introduction}

Algebraic cobordism is a theory for smooth schemes over a base scheme $S$
defined by a motivic ring spectrum $\MGL_S$ in the stable motivic homotopy
category $\SH(S)$. It is the motivic counterpart of complex cobordism $\MU$.
A famous Theorem of Quillen states that the natural map from the Lazard
ring $L_*$ classifying formal group laws to the coefficients of $\MU$
is an isomorphism, moreover $L_* \cong \integers[x_1,x_2,x_3,\ldots]$ with
$\deg(x_i)=i$ (here we divide the usual topological grading by $2$).

For an oriented motivic ring spectrum $E$ the geometric part $E_{(2,1)*}$
of the coefficients also carries a formal group law constructed in the exact same way
as in topology by evaluating the theory on $\P^\infty$ and using that $\P^\infty$
is naturally endowed with a multiplication. 

Thus there is a classifying map $L_* \to E_{(2,1)*}$. It is known
that for $E=\MGL_k$ for a field $k$ of characteristic $0$
this map is an isomorphism using the Hopkins-Morel isomorphism,
see \cite[Proposition 8.2]{hoyois.hopkins-morel}. More generally
in  \cite{levine.comparison} it is shown that over such fields the Levine-Morel
algebraic cobordism $\Omega^*(-)$ is isomorphic to $\MGL_k^{(2,1)*}(-)$
on smooth schemes over $k$.
If the base field $k$
has positive characteristic the map $L_* \to \MGL_{(2,1)*}$ becomes at least an isomorphism after
inverting the characteristic, see again \cite[Proposition 8.2]{hoyois.hopkins-morel}.

The main ingredient in the proof is that the Hopkins-Morel isomorphism yields
a computation of the slices of $\MGL_S$ with respect to Voevodsky's slice filtration,
that $\MGL_S$ is complete with respect to this filtration and that the slices
have a simple form, namely they are shifted twists of the motivic Eilenberg-MacLane spectrum.

The facts about the slices of $\MGL_S$ hold more generally true
over spectra $S$ of Dedekind domains of mixed characteristic (after inverting the positive
residue characteristics), using the motivic Eilenberg-MacLane spectrum introduced
in \cite{spitzweck.em}. The main new input of this note
is that in this case $\MGL_S$ is also complete with respect to the slice filtration (Corollary \ref{gfrerz}),
a consequence of the fact that $\MGL_S$ is connective with respect to the homotopy sheaves,
see Proposition \ref{gdrttt}.

This yields a computation of the geometric part of the homotopy groups
of $\MGL_S$ (Theorem \ref{grfe5z4}), again after inverting the residue characteristics.
In our formulation we always assume a Hopkins-Morel isomorphism for the given coefficients,
hoping that the Hopkins-Morel isomorphism will be settled completely in the future.

We prove cases of a Conjecture of Voevodsky (\cite[Conjecture 1]{voevodsky.icm}), see
Theorem \ref{hte4234t}, comparing the Lazard ring to $(\MGL_S)_{(2,1)*}$ for
$S$ the spectrum of a regular local ring.

We also give applications to some homotopy groups or sheaves of $\MGL_S$ outside the
geometric diagonal, see section \ref{dzu634e}, and discuss generalizations of our results
to motivic Landweber spectra.

\vskip.4cm

We note that the observation that the Hopkins-Morel isomorphism yields
the computation of the zero-slice of the sphere spectrum (after inverting
suitable primes), see Theorem \ref{dhu5t3}, was independently made by Oliver R\"ondigs.

\vskip.7cm

{\bf Acknowledgements:}
I would like to thank
Peter Arndt, Christian H\"ase\-meyer, Marc Hoyois, Moritz Kerz, Marc Levine, Niko Naumann,
Oliver R\"ondigs, Manfred Stelzer, Florian Strunk, J\"org Wildeshaus and Paul Arne {\O}stv{\ae}r for very helpful
discussions and suggestions on the subject.

\section{Preliminaries}

By a base scheme we always mean a separated Noetherian scheme of finite Krull dimension.
For a base scheme $S$ we let $\SH(S)$ be the stable motivic homotopy category.

We let $\MZ_S \in \SH(S)$ be the motivic Eilenberg-MacLane spectrum over $S$
constructed in \cite{spitzweck.em}. Also we let $\caM(r) \in \D(\Sh(\Sm_{S,\Zar},\integers))$
(for notation see \cite{spitzweck.em})
be the motivic complexes of weight $r \in \integers$, so as a $\GmS$-spectrum $\MZ_S$ has $\caM(r)[r]$ in
level $r$. If $S$ is the spectrum of a Dedekind domain of mixed characteristic
we note that $\caM(0)= S^0 \underline{\integers}$, thus for $X \in \Sm_S$ we have
$H^{0,0}(X,\integers)=\integers^{\pi_0(X)}$. Also $\caM(1) \cong \caO^*[-1]$,
so $H^{1,1}(X,\integers) \cong \caO^*(X)$ and $H^{2,1}(X,\integers) \cong \Pic(X)$.
We have $\caM(r) \cong 0$ for $r < 0$.

For general $S$ we denote by $\MGL_S \in \SH(S)$ the algebraic cobordism spectrum.
There is a natural map $L_* \to (\MGL_S)_{2*,*}$, where $L_*$ denotes the Lazard
ring. Fixing generators $x_i \in L_i$ there is a map
$$\Phi_S \colon \MGL_S/(x_1,x_2,\ldots)\MGL_S \to \MZ_S,$$
see \cite[\S 11.1]{spitzweck.em}, which is an isomorphism after inverting
all positive residue characteristics of $S$, see \cite[Theorem 11.3]{spitzweck.em}.

For any ring or abelian group $R$ we let $M_R \in \SH(S)$ be the Moore spectrum on $R$
and $\MR_S$ the version of $\MZ_S$ with $R$-coeffcients.

\section{Slices}

For $i \in \integers$ denote by $f_i$ resp. $l_i$ the $i$-th colocalization
resp. localization functor for Voevodsky's motivic slice filtration on $\SH(S)$.
For any $E \in \SH(S)$ and $k \ge n$ we set $E\left<n,k \right>:=l_{k+1}(f_n(E))$. Thus we have
exact triangles $$f_{k+1}(E) \to f_n(E) \to E\left<n,k \right> \to f_{k+1}(E)[1]$$
and $s_n(E) = E\left<n,n \right>$.

We note that all these functors commute with homotopy colimits.

\begin{theorem}
\label{dhu5t3}
Let $X$ be an essentially smooth scheme over a Dedekind domain of mixed characteristic
and $R$ a localization of $\integers$ such that $\Phi_X \wedge M_R$ is an isomorphism
(e.g. if every positive residue characteristic of $X$ is invertible in $R$).
Then $$s_0 M_R \cong s_0 (\MGL_X \wedge M_R) \cong \MR_X.$$
More generally $$s_n (\MGL_X \wedge M_R) \cong \Sigma^{2n,n} \MR_X \otimes L_n.$$
\end{theorem}

\begin{proof}
The first isomorphism of the first line follows from \cite[Corollary 3.3]{spitzweck.rel}.
From the assumption that $\Phi_X \wedge M_R$
is an isomorphism it follows that the map $\MGL_X \wedge M_R \to \MR_X$
induces an isomorphism on zero-slices and that $\MR_X$ is effective.
Moreover $l_1 \MZ_X \cong \MZ_X$, since negative weight motivic cohomology
vanishes in our situation. Thus
the second isomorphism of the first line follows.
The second line is a version of \cite[Theorem 4.7]{spitzweck.rel}  with $R$-coefficients.
\end{proof}

\begin{remark}
It is then also possible to determine the slices of motivic Landweber spectra
with $R$-coeffcients, see \cite{spitzweck.slices}, for example of $\KGL_X \wedge M_R$.
\end{remark}

\section{Subcategories of the stable motivic homotopy category}

Fix a base scheme $S$. We let $\SH(S)_{\ge n}$ be the $\ge n$ part (in the homological sense) of $\SH(S)$
with respect to the homotopy $t$-structure, see e.g. \cite[\S 2.1]{hoyois.hopkins-morel}. Thus $\SH(S)_{\ge n}$
is generated by homotopy colimits and extensions by the objects $\Sigma^{p,q} \Sigma^\infty_+ X$
for $X \in \Sm_S$ and $p-q \ge n$.

For each $E \in \SH(S)$ we let $\upi^\pre_{p,q}(E)$ be the presheaf
$$X \mapsto \Hom_{\SH(S)}(\Sigma^{p,q} \Sigma^\infty_+ X,E)$$ on $\Sm_S$.
Let $\upi_{p,q}(E)$ be the sheafification of $\upi^\pre_{p,q}(E)$
with respect to the Nisnevich topology. We also set $\pi_{p,q}(E):=\upi^\pre_{p,q}(E)(S)=E_{p,q}$.

We let $\SH(S)_{h \ge n}$ be the full subcategory of $\SH(S)$ of objects $E$
such that $\upi_{p,q}(E)=0$ for $p-q< n$.

\begin{lemma}
The categories $\SH(S)_{h \ge n}$ are closed under homotopy colimits and extensions
in $\SH(S)$.
\end{lemma}

\begin{proof}
Th functors $\upi_{p,q}$ respect sums. Moreover the long exact sequences of homotopy sheaves
associated to an exact triangle in $\SH(S)$ show that $\SH(S)_{h \ge n}$ is closed under
cofibers and extensions. This shows the claim.
\end{proof}

\begin{proposition}
\label{gegrsy}
Let $i \colon Z \hookrightarrow S$ be a closed inclusion of base schemes.
Then $i_*(\SH(Z)_{h \ge n}) \subset \SH(S)_{h \ge n}$.
\end{proposition}

\begin{proof}
Let $E \in \SH(Z)_{h \ge n}$.
Let $Y$ be the spectrum of the henselization of a local ring of a scheme from $\Sm_S$.
Then $Y_Z:=Y \times_S Z$ is also the spectrum of a henselian local ring, and
$\upi^\pre_{p,q}(i_*E)(Y) \cong \upi^\pre_{p,q}(E)(Y_Z)=0$ for $p-q < n$
(the first isomorphism holds since $i_*$ commutes with homotopy colimits).
\end{proof}

We let $\SH^\se_s(S)$ be the homotopy category of presheaves of $\se$-spectra on $\Sm_S$ localized
with respect to the Nisnevich topology, and $\SH^\se(S)$ the further $\bA^1$-localization of that category.

We let $\SH^\se_s(S)_{\ge n}$ be the $\ge n$ part (in the homological sense)
of $\SH^\se_s(S)$ with respect to the standard $t$-structure, and for $E \in \SH^\se_s(S)$ we let $E_{\ge n}$ and
$E_{\le n}$ be the corresponding truncations. We let $E_{=n} := (E_{\ge n})_{\le n}$.

As above for $E \in \SH^\se_s(S)$ we have the presheaves $\upi^\pre_p(E)$ and the sheaves
$\upi_p(E)$. For $E \in \SH^\se_s(S)$ we have $E \in \SH^\se_s(S)_{\ge n}$ if and only
if $\upi_k(E)=0$ for $k < n$.

Note that $\SH^\se_s(S)_{\ge n}$ is generated by homotopy colimits and extensions by
the objects $\Sigma^n \Sigma^\infty_+ X$, $X \in \Sm_S$, thus the canonical
functor $\sigma \colon \SH^\se_s(S) \to \SH(S)$ sends $\SH^\se_s(S)_{\ge n}$ to $\SH(S)_{\ge n}$.

\begin{lemma}
We have $\SH(S)_{h \ge n} \subset \SH(S)_{\ge n}$. If $S$ is the spectrum of a field
then the inclusion is an equality.
\end{lemma}

\begin{proof}
Let $E \in \SH(S)_{h \ge n}$. For any $i \in \naturals$ let $E_i$ be the image
of $\Sigma^{i,i} E$ in $\SH^\se_s(S)$. By assumption we have $E_i \in \SH^\se_s(S)_{\ge n}$.
Thus $\Sigma^{-i,-i} \sigma(E_i) \in \SH(S)_{\ge n}$. The proof of the first statement concludes by noting
that $E \cong \hocolim_{i \to \infty} \Sigma^{-i,-i} \sigma(E_i)$.

The second statement is \cite[Theorem 2.3]{hoyois.hopkins-morel}.
\end{proof}

\begin{lemma}
\label{hterge}
Let $E \in \SH^\se_s(S)$. Then $E \to \holim_{n \to \infty} E_{\le n}$ is an isomorphism.
\end{lemma}

\begin{proof}
We show that for all $n \in \integers$ we have $\upi_n(E) \cong \upi_n(\holim_{k \to \infty} E_{\le k})$.
Fix $n \in \integers$ and let $X \in \Sm_S$ be of dimension $d$. We are ready if we show
$$\upi_n(E)|_{X_\Nis} \cong \upi_n(\holim_{k \to \infty} E_{\le k})|_{X_\Nis} \; \; (*).$$ For $m > n+d$
we have $\upi^\pre_n(E_{=m}[j])(Y)=0$ for $Y \in X_\Nis$ and $j \ge 0$, so homing out of $\Sigma^\infty_+ Y$ into
the exact triangle
$$E_{=m} \to E_{\le m} \to E_{\le (m-1)} \to E_{=m}[1]$$
shows that $\upi^\pre_n(E_{\le m})(Y) \cong \upi^\pre_n(E_{\le (m-1)})(Y)$.
Using the Milnor exact sequence this shows that
$$\upi^\pre_n(\holim_{k \to \infty} E_{\le k})|_{X_\Nis} \cong \upi^\pre_n(E_{\le m})|_{X_\Nis}$$
for $m \ge n+d$. Sheafifiying proves $(*)$.
\end{proof}

\begin{corollary}
\label{dwrzjj}
Let $$\cdots \to E_{i+1} \to E_i \to E_{i-1} \to \cdots \to E_1 \to E_0$$
be an inverse system of objects in $\SH(S)$. Suppose for each $n \in \naturals$ there is an $N \in \naturals$
such that $E_i \in \SH(S)_{h \ge n}$ for $i \ge N$. Then $\holim_{i \to \infty} E_i \cong 0$.
\end{corollary}

\begin{proof}
Fix $q \in \integers$ and let $F_i$ be the image of $\Sigma^{q,q} E_i$ in $\SH^\se_s(S)$.
We are ready if we show $\holim_{i \to \infty} F_i \cong 0$.
By assumption for every $n \in \naturals$ there is a $N \in \naturals$ such that $F_i \in \SH^\se_s(S)_{\ge n}$ for
each $i \ge N$. By Lemma \ref{hterge} we have $F_i \cong \holim_{k \to \infty} (F_i)_{\le k}$.
Thus $$\holim_{i \to \infty} F_i \cong \holim_i \holim_k (F_i)_{\le k} \cong \holim_k \holim_i (F_i)_{\le k}
\cong \holim_k 0 \cong 0.$$
\end{proof}

We also have the

\begin{corollary}
Let $E \in \SH(S)_{h \ge n}$ and $X \in \Sm_S$ of dimension $d$. Then
$$\upi^\pre_{p,q}(E)(X)=0$$ for $p-q<n-d$.
\end{corollary}

\begin{proposition}
\label{gr455}
Let $$\cdots \to E_{i+1} \to E_i \to E_{i-1} \to \cdots \to E_1 \to E_0$$
be an inverse system of objects in $\SH(S)_{h \ge n}$.
Suppose for each $p,q \in \integers$ and $d \in \naturals$ there is an $N \in \naturals$
such that for $X \in \Sm_S$ of dimension $d$ the map
$$\upi^\pre_{p,q}(E_{i+1})(X) \to \upi^\pre_{p,q}(E_i)(X)$$
is an isomorphism for all $i \ge N$. Then $\holim_{i \to \infty} E_i \in \SH(S)_{h \ge n}$.
(Here the latter homotopy limit is computed in $\SH(S)$.)
\end{proposition}

\begin{proof}
Let $p,q \in \integers$, $d \in \naturals$ and $X \in \Sm_S$ of dimension $d$.
Choose $N \in \naturals$ such that for any $Y \in \Sm_S$ of dimension $\le d$
the map
$$\upi^\pre_{p,q}(E_{i+1})(Y) \to \upi^\pre_{p,q}(E_i)(Y)$$
is an isomorphism for all $i \ge N$.
We claim that $$\upi_{p,q}(\holim_k E_k)|_{X_\Nis} \cong \upi_{p,q}(E_i)|_{X_\Nis}$$
for all $i \ge N$. For every $Y \in X_\Nis$ we have the Milnor short exact sequence
$$0 \to \text{$\lim_i$}^1 \upi^\pre_{p+1,q}(E_i)(Y) \to \upi^\pre_{p,q}(\holim_i E_i)(Y) \to \lim_i \upi^\pre_{p,q}(E_i)(Y) \to 0.$$
The $\lim^1$-term vanishes because the inverse system of abelian groups stabilizes by assumption.
Sheafifying we see that $\upi_{p,q}(\holim_k E_k)|_{X_\Nis} \cong \upi_{p,q}(E_i)|_{X_\Nis}$ for $i \ge N$,
in particular $\upi_{p,q}(\holim_k E_k)|_{X_\Nis} = 0$ in case $p-q<n$. Since this is true for all $X \in \Sm_S$
we conclude $\upi_{p,q}(\holim_k E_k) =0$ for $p-q < n$.
\end{proof}

\section{Connectivity of algebraic cobordism}

\begin{lemma}
\label{vftjt}
Let $X$ be a smooth scheme over a Dedekind domain of mixed characteristic or over a field.
Then for any abelian group $A$ we have $\MA_X \in \SH(X)_{h \ge 0}$.
\end{lemma}

\begin{proof}
This follows from the fact that the motivic complexes $\caM(r)$ have vanishing
$i$-th cohomology sheaf for $i>r$, see \cite[Corollary 4.4]{geisser.dede}.
\end{proof}

\begin{proposition}
\label{fhtrth}
Let $S$ be the spetrum of a discrete valuation ring of mixed characteristic, $j \colon \eta \to S$
the inclusion of the generic point. Then for any abelian group $A$ we have
$j_* \MA_\eta \in \SH(S)_{h \ge 0}$.
\end{proposition}

\begin{proof}
Let $i \colon s \to S$ be the inclusion of the closed point.
We have an exact triangle
$$i_!i^! \MA_S \to \MA_S \to j_* \MA_\eta \to i_!i^! \MA_S[1]$$
and an isomorphism $i^! \MA_S \cong \MA_s(-1)[-2] \in \SH(s)_{h \ge -1}$, see \cite[Theorem 7.4]{spitzweck.em}.
We conclude with Proposition \ref{gegrsy} and Lemma \ref{vftjt}.
\end{proof}

\begin{lemma}
\label{gfewet}
Let the situation be as in Proposition \ref{fhtrth}.
Then $$j_*\MGL_\eta\left<0,n\right> \wedge M_A \in \SH(S)_{h \ge 0}$$
for all $n \ge 0$.
\end{lemma}

\begin{proof}
We can assume $A=\integers$. Since $\eta$ is of characteristic $0$
we have $s_n \MGL_\eta \cong \Sigma^{2n,n} \MZ \otimes L_n$. Moreover we have exact triangles
$$s_n \MGL_\eta \to \MGL_\eta\left<0,n\right> \to \MGL_\eta \left<0,n-1 \right> \to s_n \MGL_\eta [1].$$
Applying $j_*$ to these triangles and using Proposition \ref{fhtrth} one concludes by induction on $n$.
\end{proof}

\begin{lemma}
\label{rergr}
Let the situation be as in Proposition \ref{fhtrth}.
Let $p,q \in \integers$ and $X \in \Sm_S$ of dimension $d$.
Then $$\upi^\pre_{p,q}(j_* \MGL_\eta\left< 0,n+1 \right>)(X) \to \upi^\pre_{p,q}(j_* \MGL_\eta\left< 0,n \right>)(X)$$
is an isomorphism for $n \ge p-q+d$.
\end{lemma}

\begin{proof}
Consider the exact triangle
$$j_* s_{n+1} \MGL_\eta \to j_* \MGL_\eta \left< 0,n+1 \right> \to j_* \MGL_\eta \left< 0,n \right> \to s_{n+1} \MGL_\eta [1].$$
We have $$\upi^\pre_{p,q}(j_* s_{n+1} \MGL_\eta)(X) = H_{\mathrm{mot}}^{2(n+1)-p,n+1-q}(X_\eta, L_{n+1}).$$
The latter group vanishes for $2(n+1)-p> n+1 -q+d$, showing the claim.
\end{proof}

\begin{lemma}
\label{eweggg}
Let the situation be as in Proposition \ref{fhtrth}.
Then $j_* \MGL_\eta \in \SH(S)_{\h \ge 0}$.
\end{lemma}

\begin{proof}
Consider the inverse system
$$\cdots \to j_* \MGL_\eta \left<0,n+1 \right> \to j_* \MGL_\eta \left< 0,n \right> \to \cdots \to j_* s_0 \MGL_\eta$$
in $\SH(S)$.
Since $j_*$ preserves homotopy limits the homotopy limit over this system is $j_* \MGL_\eta$,
using \cite[Corollary 2.4 and Lemma 8.10 or Theorem 8.12]{hoyois.hopkins-morel}.
By Lemma \ref{gfewet} every object of this system is in $\SH(S)_{h \ge 0}$.
Moreover by Lemma \ref{rergr} the assumptions of Proposition \ref{gr455} are satisfied.
Thus this Proposition implies the claim.
\end{proof}

\begin{proposition}
\label{greerg}
Let the situation be as in Proposition \ref{fhtrth} and let $i \colon s \to S$ be the inclusion
of the closed point. Then $i^! \MGL_S \in \SH(s)_{\ge -1}$.
\end{proposition}

\begin{proof}
Note first that $i^*$ sends $\SH(S)_{\ge 0}$ to $\SH(s)_{\ge 0}$.
We have $\MGL \in \SH(S)_{\ge 0}$ and by Lemma \ref{eweggg} also $j_* \MGL_\eta \in \SH(S)_{h \ge 0}
\subset \SH(S)_{\ge 0}$.
Applying $i^*$ to the exact triangle
$$i_!i^! \MGL_S \to \MGL_S \to j_* \MGL_\eta \to i_!i^! \MGL_S [1]$$
shows the claim.
\end{proof}

\begin{lemma}
\label{gferghd}
Let $S$ be the spectrum of a discrete valuation ring of mixed characteristic.
Then $\MGL_S \in \SH(S)_{h \ge -1}$.
\end{lemma}

\begin{proof}
Let the notation be as above.
The claim follows from the exact triangle
$$i_!i^! \MGL_S \to \MGL_S \to j_* \MGL_\eta \to i_!i^! \MGL_S [1],$$
Lemma \ref{eweggg}, Proposition \ref{greerg} and Proposition \ref{gegrsy}.
\end{proof}

\begin{proposition}
\label{gdrttt}
Let $S$ be the spectrum of a Dedekind domain of mixed characteristic.
Then $\MGL_S \in \SH(S)_{h \ge -1}$.
\end{proposition}

\begin{proof}
The henselization of a local ring of a scheme in $\Sm_S$ lies over a local ring
of $S$, thus the claim follows from Lemma \ref{gferghd}.
\end{proof}

Compare the following result to \cite[Conjecture 15]{voevodsky.open}.

\begin{corollary}
\label{gfrerz}
Let $S$ be the spectrum of a Dedekind domain of mixed characteristic.
Let $R$ be a localization of $\integers$ such that $\Phi_S \wedge M_R$
is an isomorphism. Then for any $R$-module $A$ we have
$$f_n \MGL_S \wedge M_A \cong \holim_{k \to \infty} \MGL_S \left< n,k \right> \wedge M_A.$$
\end{corollary}

\begin{proof}
Under the assumption we have $\f_k \MGL_S \wedge M_A \in \SH(S)_{h \ge k-1}$, since this is a homotopy
colimit of objects of the form $\Sigma^{2i,i} \MGL_S \wedge M_A$ with $i \ge k$, see the proof
of \cite[Theorem 4.7]{spitzweck.rel}, using Proposition \ref{gdrttt}. Thus by Corollary \ref{dwrzjj} we have
$\holim_{k \to \infty} f_k \MGL_S \wedge M_A \cong 0$ implying the claim.
\end{proof}

\begin{remark}
A similar result holds for motivic Landweber spectra using the same argument as in the proof
of \cite[Lemma 8.11]{hoyois.hopkins-morel}. For example $\KGL_S \wedge M_A$ is complete with respect to
the slice filtration.
\end{remark}

\section{The geometric part of algebraic cobordism}

\begin{lemma}
\label{jewet45}
Let $S$ be the spectrum of a Dedekind domain of mixed characteristic.
Let $R$ be a localization of $\integers$ such that $\Phi_S \wedge M_R$
is an isomorphism. Let $p,q \in \integers$ and $X \in \Sm_S$.
Then for any $R$-module $A$ the inverse system of abelian groups
$(\upi^\pre_{p,q}(\MGL_S \left< 0,k \right> \wedge M_A)(X))_k$ eventually becomes
constant for $k \to \infty$.
\end{lemma}

\begin{proof}
This follows from the exact triangle
$$s_k \MGL_S \wedge M_A \to \MGL_S \left<0,k \right> \wedge M_A \to \MGL_S \left< 0,k-1 \right> \wedge M_A \to s_k \MGL_S \wedge M_A [1]$$
and $s_k \MGL_S \wedge M_A \cong \Sigma^{2k,k} \MA \otimes L_k$
since $\upi^\pre_{p,q}(\Sigma^{2k+j,k} \MA)(X)=0$, $j \ge 0$, for $k$ big enough.
\end{proof}

\begin{corollary}
\label{dsgtjt}
Let $S$ be the spectrum of a Dedekind domain of mixed characteristic.
Let $R$ be a localization of $\integers$ such that $\Phi_S \wedge M_R$
is an isomorphism. Let $p,q \in \integers$ and $X \in \Sm_S$.
Then for any $R$-module $A$
the canonical map $$\upi^\pre_{p,q}(\MGL_S \wedge M_A)(X) \to \lim_k \upi^\pre_{p,q}(\MGL_S \left<0,k \right> \wedge M_A)(X)$$
is an isomorphism.
\end{corollary}

\begin{proof}
This follows from Corollary \ref{gfrerz}, the Milnor short exact sequence and Lemma \ref{jewet45}.
\end{proof}

\begin{lemma}
\label{xfhtjt}
Let $S$ be the spectrum of a Dedekind domain of mixed characteristic.
Let $R$ be a localization of $\integers$ such that $\Phi_S \wedge M_R$
is an isomorphism. Let $n \in \integers$. Then for $k \ge n+1$ and any $R$-module $A$ the natural map
$$\pi_{2n,n} \MGL_S \left<n,k+1 \right> \wedge M_A \to \pi_{2n,n} \MGL_S \left<n,k \right> \wedge M_A$$
is an isomorphism.
\end{lemma}

\begin{proof}
This follows from the exact sequence
$$\pi_{2n,n} s_{k+1} \MGL_S \wedge M_A \to \pi_{2n,n} \MGL_S \left<n,k+1 \right> \wedge M_A \to$$
$$\pi_{2n,n} \MGL_S \left<n,k \right> \wedge M_A \to
\pi_{2n-1,n} s_k \MGL_S \wedge M_A$$
and the fact the the two outer terms are $0$ for $k \ge n+1$.
\end{proof}

\begin{corollary}
\label{sgtjtr}
Let $S$ be the spectrum of a Dedekind domain of mixed characteristic.
Let $R$ be a localization of $\integers$ such that $\Phi_S \wedge M_R$
is an isomorphism.
Then for any $R$-module $A$ the canonical map $$\pi_{2n,n} \MGL_S \wedge M_A \to \pi_{2n,n} \MGL_S \left< n,n+1 \right> \wedge M_A$$ is an isomorphism.
\end{corollary}

\begin{proof}
This follows from Corollary \ref{dsgtjt} and Lemma \ref{xfhtjt}.
\end{proof}

\begin{theorem}
\label{grfe5z4}
Let $S$ be the spectrum of a Dedekind domain of mixed characteristic.
Let $R$ be a localization of $\integers$ such that $\Phi_S \wedge M_R$
is an isomorphism.
Then for every $n \in \integers$ and $R$-module $A$ there is a canonical isomorphism
$$\pi_{2n,n} \MGL_S \wedge M_A \cong L_n \otimes A \oplus L_{n+1} \otimes \mathrm{Pic}(S) \otimes A.$$ 
\end{theorem}

\begin{proof}
We have the exact sequence
$$\pi_{2n+1,n} s_n \MGL_S \wedge M_A \to \pi_{2n,n} s_{n+1} \MGL_S \wedge M_A \to \pi_{2n,n} \MGL_S \left<n,n+1 \right> \wedge M_A$$
$$\to \pi_{2n,n} s_n \MGL_S \wedge M_A \to \pi_{2n-1,n} s_{n+1} \MGL_S \wedge M_A.$$
The two outer terms are $0$. Also $\pi_{0,0} \Sigma^{2,1} \MA_S \cong \mathrm{Pic}(S) \otimes A$.
Moreover there is a canonical map
$L_n \otimes A \to \pi_{2n,n} \MGL_S \wedge M_A$ splitting the resulting short exact sequence,
whence the claim follows form Corollary \ref{sgtjtr}.
\end{proof}

\begin{corollary}
Let $S$ be the spectrum of a Dedekind domain of mixed characteristic and
$R$ the localization of $\integers$ obtained by inverting all positive residue characteristics
of $S$. Then $$(\pi_{2n,n} \MGL_S) \otimes R \cong (L_n \oplus L_{n+1} \otimes \mathrm{Pic}(S)) \otimes R.$$
\end{corollary}

We have the following case of a Conjecture of Voevodsky (see \cite[Conjecture 1]{voevodsky.icm}):

\begin{theorem}
\label{hte4234t}
Let $S=\Spec(R)$, where $R$ is a (regular) Noetherian local ring which is regular over some discrete valuation ring
of mixed characteristic. Then the natural map
$$L_* \to (\MGL_S)_{2*,*}$$
becomes an isomorphism after inverting the residue characteristic of the closed point of $S$.
\end{theorem}

\begin{proof}
By Popescu's Theorem $R$ is a filtered colimit of smooth algebras over a discrete valuation ring $V$
of mixed characteristic. Thus we are reduced to the case where $R$ is the local ring
of a scheme $X \in \Sm_{\Spec(V)}$ by a colimit argument. Let $p$ be the residue characteristic
of the closed point of $\Spec(V)$.
By the same type of argument as above and the vanishing of $(p,q)$-motivic
cohomology of such local rings for $p>q$ we have
$$(MGL_S)_{2n,n} [1/p] \cong (s_n \MGL_S)_{2n,n} [1/p] \cong L_n [1/p],$$
using that for a fixed dimension only a fixed finite number of slices of $\MGL_S[1/p]$
contribute to the value of $\pi^\pre_{2n,n}(\MGL_S [1/p])$ on schemes of that dimension.
\end{proof}

More generally we have

\begin{proposition}
Let $S$ be as in the previous Theorem and $E \in \SH(S)$ a motivic Landweber spectrum modelled
on $E^{\mathrm{top}}_{2*}$. Then the natural map
$$E^{\mathrm{top}}_{2*} \to E_{2*,*}$$ is an isomorphism after
inverting the residue characteristic of the closed point of $S$.
\end{proposition}

\begin{proof}
This follows from the definition of motivic Landweber spectrum.
\end{proof}

\section{Some other parts of algebraic cobordism}
\label{dzu634e}

We have the following vanishing result:

\begin{proposition}
Let $S$ be the spectrum of a Dedekind domain of mixed characteristic.
Let $R$ be a localization of $\integers$ such that $\Phi_S \wedge M_R$
is an isomorphism.
Then for any $p,q \in \integers$ and $R$-module $A$ we have
$\upi_{p,q} \MGL_S \wedge M_A \cong 0$ for $p< 2q$ or $p<q$. In particular
we have $\MGL_S \wedge M_R \in \SH(S)_{h \ge 0}$.
\end{proposition}

\begin{proof}
Let $p,q \in \integers$ satisfying the condition of the statement.
Let $d \in \naturals$. Then there is a $N \ge q$ such that for any scheme of dimension $\le d$
and $k \ge N$ the map $$\upi^\pre_{p,q}(\MGL_S \wedge M_A)(X) \to
\upi^\pre_{p,q}(\MGL_S\left< 0,k \right> \wedge M_A)(X)$$
is an isomorphism. The assertion then follows by an induction argument on $i$ showing
that $\upi_{p,q}(\MGL_S \left< q,q+i \right> \wedge M_A)=0$.
\end{proof}

Generalizing the argument given in the last proof we get

\begin{lemma}
\label{dh5t43}
Let $S$ be the spectrum of a Dedekind domain of mixed characteristic.
Let $R$ be a localization of $\integers$ such that $\Phi_S \wedge M_R$
is an isomorphism.
Then for any $p,q \in \integers$ and $R$-module $A$ we have
$$\upi_{p,q}(\MGL_S \wedge M_A) \cong \lim_k \upi_{p,q}(\MGL_S \left<0,k \right> \wedge M_A)
\cong \upi_{p,q}(\MGL_S \left<\max(0,q),n \right> \wedge M_A)$$
for $n \ge p-q$ or $n \ge p-2q$.
\end{lemma}

\begin{corollary}
Let $S$ be the spectrum of a Dedekind domain of mixed characteristic.
Let $R$ be a localization of $\integers$ such that $\Phi_S \wedge M_R$
is an isomorphism. Then for any $R$-module $A$
and $n \in \integers$ we have $\upi_{n,n}(\MGL_S \wedge M_A) \cong \underline{K}^M_{-n} \otimes A$,
where $\underline{K}^M_{-n}$ is the $(-n)$-th Milnor-$K$-theory sheaf defined via the
degree $(-n,-n)$-motivic cohomology.
\end{corollary}

\begin{corollary}
Let $S$ be the spectrum of a Dedekind domain of mixed characteristic.
Let $R$ be a localization of $\integers$ such that $\Phi_S \wedge M_R$
is an isomorphism. Then for any $R$-module $A$
and $n \in \integers$ we have $\upi_{2n,n}(\MGL_S \wedge M_A) \cong \underline{L}_n \otimes A$,
where the latter sheaf is the constant sheaf on $L_n \otimes A$.
\end{corollary}

\begin{corollary}
Let $S$ be the spectrum of a Dedekind domain of mixed characteristic.
Let $R$ be a localization of $\integers$ such that $\Phi_S \wedge M_R$
is an isomorphism. Then for any $R$-module $A$
and $n \in \integers$ we have $\upi_{2n+1,n}(\MGL_S \wedge M_A) \cong \caO^* \otimes L_{n+1} \otimes A$.
\end{corollary}

\begin{proof}
By Lemma \ref{dh5t43} we have $$\upi_{2n+1,n}(\MGL_S \wedge M_A) \cong
\upi_{2n+1,n}(\MGL_S \left< n,n+1 \right> \wedge M_A).$$ The long exact sequence
of sheaves associated to the exact triangle
$$s_{n+1} \MGL_S \wedge M_A \to \MGL_S \left< n,n+1 \right> \wedge M_A \to s_n \MGL_S \wedge M_A
\to s_{n+1} \MGL_S \wedge M_A [1]$$ together with
$$\upi_{2n+1,n}(s_n \MGL_S \wedge M_A [-1])= \upi_{2n+1,n}(s_n \MGL_S \wedge M_A)=0$$ and
$$\upi_{0,0}(\Sigma^{1,1} \MA_S) \cong \caO^* \otimes A$$
gives the result.
\end{proof}

Similarly we get

\begin{corollary}
Let $S$ be the spectrum of a Dedekind domain of mixed characteristic.
Let $R$ be a localization of $\integers$ such that $\Phi_S \wedge M_R$
is an isomorphism. Then for any $R$-module $A$
and $n \in \integers$ there is an exact sequence
$$\underline{K}^M_{1-n} \otimes A \to \upi_{n+1,n}(\MGL_S \wedge M_A) \to \caH_{\mathrm{mot}}^{-n-1,-n}(-,A) \to 0,$$
where the latter group denotes the motivic cohomology sheaf in degees $(-n-1,-n)$ and $A$-coefficients.
\end{corollary}

We also have

\begin{proposition}
Let $S$ be the spectrum of a Dedekind domain of mixed characteristic.
Let $R$ be a localization of $\integers$ such that $\Phi_S \wedge M_R$
is an isomorphism. Then for any $R$-module $A$
and $n \in \integers$ there is an exact sequences
$$H^{3,2}(S) \otimes A \otimes L_{n+2} \to \pi_{2n+1,n} \MGL_S \to H^{1,1}(S,A) \otimes L_{n+1} \to 0.$$
If $A$ is torsionfree the first map is also injective.
\end{proposition}

\begin{proposition}
Let $S$ be the spectrum of a Dedekind domain of mixed characteristic.
Let $R$ be a localization of $\integers$ such that $\Phi_S \wedge M_R$
is an isomorphism. Let $X$ be an essentially smooth scheme over $S$.
Then for any $R$-module $A$,
$n \in \integers$ and $i \ge 2$ we have $\MGL_S^{2n+i,n}(X,A)=0$.
\end{proposition}

\begin{proof}
This follows from the above considerations and the fact that for $X \in \Sm_S$
we have $H^{p,q}(X)=0$ for $p \ge 2q+2$, since motivic cohomology is computed
as hypercohomology over $S$ of the Bloch-Levine cycle complexes.
\end{proof}

We leave assertions about the groups $\pi_{2n-1,n} \MGL_S \wedge M_A$,
$\pi_{n,n} \MGL_S \wedge M_A$ and $\pi_{n-1,n} \MGL_S \wedge M_A$ to the interested reader.

\bibliographystyle{plain}
\bibliography{ma}

\vspace{0.1in}

\begin{center}
Fakult{\"a}t f{\"u}r Mathematik, Universit{\"a}t Osnabr\"uck, Germany.\\
e-mail: markus.spitzweck@uni-osnabrueck.de
\end{center}

\end{document}